\documentclass[a4,12pt]{article}%

\usepackage{graphicx}
\usepackage{graphics}

\usepackage{amsmath}
\usepackage{amsfonts}
\usepackage{amssymb}
\usepackage{enumerate}
\newtheorem{theorem}{Theorem}[section]

\newtheorem{definition}[theorem]{Definition}
\newtheorem{example}[theorem]{Example}

\newtheorem{lemma}[theorem]{Lemma}
\newtheorem{assumption}[theorem]{Assumption}

\newtheorem{remark}[theorem]{Remark}

\newenvironment{proof}[1][Proof]{\textbf{#1.} }{\ \rule{0.5em}{0.5em}}
\newcommand{\E}{{\rm \bf E}}

\newcommand{\prob}{{\rm \bf P}}

\newcommand{\lcp}{{\rm LCP}}

\newcommand{\conv}{{\rm conv}}
\newcommand{\dN}{{\mathbb N}}

\newcommand{\dR}{{\mathbb R}}

\newcommand{\ep}{\varepsilon}

\newcounter{figurecounter}
\begin{document}

\title{Quitting Games and Linear Complementarity Problems%
\thanks{The authors thank Nimrod Megiddo for noticing the connection between our game theoretic condition and $Q$-matrices,
and Efim Gluskin, Boaz Klartag, Orin Munk, and Arik Tamir for useful discussions.
The first author acknowledges the support of the Israel Science Foundation, Grant \#323/13.}}

\author{Eilon Solan and Omri N.~Solan%
\thanks{The School of Mathematical Sciences, Tel Aviv
University, Tel Aviv 6997800, Israel. e-mail: eilons@post.tau.ac.il, omrisola@post.tau.ac.il.}}

\maketitle

\begin{abstract}
We prove that every multiplayer quitting game admits a sunspot $\ep$-equilibrium for every $\ep > 0$,
that is, an $\ep$-equilibrium in an extended game in which the players observe a public signal at every stage.
We also prove that if a certain matrix that is derived from the payoffs in the game is a $Q$-matrix in the sense of linear complementarity problems,
then the game admits a Nash $\ep$-equilibrium for every $\ep > 0$.
\end{abstract}


\section{Introduction}

Shapley (1953) introduced the model of stochastic games as a model of dynamic interactions in which
players' actions affect both the stage payoffs and the evolution of the state variable.
Shapley studied the two-player zero-sum model, and proved that the discounted value always exists and that both players have stationary optimal strategies.
This result was extended to the existence of discounted equilibria in multiplayer stochastic games by Fink (1964) and Takahashi (1964).

The equilibrium strategies are not robust, as they depend on the discount factor.
Mertens and Neyman (1981) proposed a solution concept that is robust to variation in the discount factor:
given $\ep > 0$, a strategy profile is a \emph{uniform $\ep$-equilibrium}
if it is a discounted $\ep$-equilibrium for every discount factor sufficiently close to 0.
Mertens and Neyman (1981) proved that every two-player zero-sum stochastic game admits a uniform $\ep$-equilibrium, for every $\ep > 0$,
and Vieille (2000a,2000b) extended this result to two-player non-zero-sum games.
Solan (1999) proved that a uniform $\ep$-equilibrium exists in three-player absorbing games,
which are stochastic games with a single non-absorbing states.
It is still not known whether any multiplayer stochastic game admits a uniform $\ep$-equilibrium, for every $\ep > 0$.

In their study of uniform equilibrium in multiplayer stochasic games,
Solan and Vieille (2001) introduced a new class of absorbing games, called \emph{quitting games},
which is inspired by the game studied by Flesch, Thuijsman, and Vrieze (1997).
In a quitting game, each one of $N$ players decides at every stage whether to continue or to quit.
As long as all players continue, the game continues.
Once at least one player quits, the game terminates, and the terminal payoff depends on the set of players who decide to quit at the terminal stage.
Solan and Vieille (2001) proved that if each player prefers to quit alone rather than to quit with other players,
then a uniform $\ep$-equilibrium exists.
This result was extended to a more general class of quitting games by Simon (2012).

Aumann (1974, 1987) introduced the concept of correlated equilibrium in strategic-form games.
A \emph{correlated equilibrium} in a strategic-form game is an equilibrium in an extended game that includes a correlation device,
which sends to each player a private signal before the play starts.
In dynamic games several variations of correlated equilibrium come to mind (see Forges (1986) and Fudenberg and Tirole (1991)).
The most general concept is \emph{communication equilibrium}, that corresponds to an equilibrium in an extended game
in which at every stage the device receives a private message from each player and sends a private signal to each player,
which may depend on past messages and signals of all players.
A more restricted concept is \emph{extensive-form correlated equilibrium},
which corresponds to the situation in which the device sends a private message to each player at the beginning of every stage,
and does not receive any message from the players.
Yet a more restricted concept is \emph{normal-form correlated equilibrium},
which corresponds to the situation in which the device sends one private message to each player at the beginning of the game.
Cass and Shell (1984) proposed the concept of \emph{sunspot equilibrium},
which is an equilibrium in a game extended by a correlation device that publicly sends to the players at every stage
a uniformly distributed random variable in $[0,1]$ that is chosen independently of past signals and past play.

Solan and Vieille (2002) proved that every multiplayer stochastic game admits an extensive-form uniform correlated $\ep$-equilibrium, for every $\ep > 0$,
and Solan and Vohra (2002) proved that every multiplayer absorbing game admits a normal-form uniform correlated $\ep$-equilibrium, for every $\ep > 0$.
In this paper we prove that every multiplayer quitting game admits a sunspot uniform $\ep$-equilibrium, for every $\ep > 0$.
Solan and Vieille (2001) proved that in quitting games, every undiscounted $\ep$-equilibrium is also uniform,
and their argument carries over to correlated equilibria.

In this paper we prove that every multiplayer quitting game
admits an undiscounted sunspot $\ep$-equilibrium, for every $\ep > 0$.
By the above mentioned result of Solan and Vieille (2001), this implies
that every multiplayer quitting game admits a sunspot uniform $\ep$-equilibrium, for every $\ep > 0$.

Our proof uses heavily the notion of $Q$-matrices from linear complementarity problems (see, e.g., Murty, 1988).
Given an $n \times n$ matrix $R$ and a vector $q \in \dR^n$,
the \emph{linear complementarity problem} $\lcp(R,q)$ is the problem of finding two vectors
$w,z \in \dR^n_{\geq 0} := \{x \in \dR^n \colon x_i \geq 0, \ \ \ \forall i \in [n]\}$
such that (a) $w=Rz+q$ and (b) $w_i = 0$ or $z_i = 0$ for every $i \in [n]$.
A matrix $R$ is called a \emph{$Q$-matrix} if a solution to the problem $\lcp(R,q)$ exists for every $q \in \dR^n$.

Denote by $r^i \in \dR^n$ the terminal payoff in the quitting game if player~$i$ quits alone.
By adding a constant to the payoffs, we can assume without loss of generality that $r^i_i=0$; that is,
each player who quits alone receives 0.
Call a player \emph{normal} if there is a player $j \neq i$ such that $r^j_i \leq 0$:
if player~$i$ knows that player~$j$ is going to quit alone,
he will prefer to quit before him.%
\footnote{Our notion of normal players will be weaker than the one provided here; see Section~\ref{section:normal}.}
Let $I_*$ be the set of all normal players,
and let $\widehat r^i \in \dR^{I_*}$ be the vector $r^i$ restricted to the coordinates that correspond to normal players.
Rename the players so that the normal players are players number $1,2,\cdots,|I_*|$,
and let $\widehat R$ be the $|I_*| \times |I_*|$ matrix whose $i$'th column coincides with $\widehat r^i$.

We show that if the matrix $\widehat R$ is not a $Q$-matrix,
then the quitting game admits a uniform $\ep$-equilibrium for every $\ep > 0$,
and if the matrix $\widehat R$ is a $Q$-matrix,
then for every $\ep > 0$ the quitting game admits a uniform sunspot $\ep$-equilibrium,
in which at every stage at most one player quits with positive probability.

Our contribution is threefold.
First, we prove that every quitting game admits a uniform sunspot $\ep$-equilibrium.
Second, we identify a general condition that ensures that a uniform $\ep$-equilibrium exists in quitting games.
Third, we relate the question of existence of uniform equilibrium in stochastic games to linear complementarity problems.
In particular, our work limits the class of quitting games for which the existence of a uniform $\ep$-equilibrium is not known.
We hope that our result will pave the road to proving the existence of a sunspot $\ep$-equilibrium in every stochastic game
and the existence of a uniform $\ep$-equilibrium in quitting games.

The paper is organized as follows.
The model and the main results are described in Section~\ref{sec:model}.
The proof of the main result appears in Section~\ref{sec:proof}.
In Section~\ref{section:characterization} we discuss the characterization of sunspot equilibrium payoffs when the matrix $\widehat R$ is a $M$-matrix,
namely, a $Q$-matrix for which in each row and each column there is exactly one positive entry.
Discussion and open problems, including a discussion on the extension of our result to stopping games, appear in Section~\ref{sec:discussion}.

\section{The Model and the Main Results}
\label{sec:model}

\subsection{The Model}

A \emph{quitting game} $\Gamma((r^S)_{S \subseteq I})$ is given by
\begin{itemize}
\item   A finite set of players $I = [N] := \{1,2,\cdots,N\}$.
\item   For every subset $S$ of $I$, a vector $r^S \in [-1,1]^N$.
\end{itemize}

Note that we assume w.l.o.g.~that payoffs are bounded by 1.

The game evolves as follows.
At every stage $t \in \dN$ each player decides whether to continue or quit.
Denote by $t^*$ the first stage in which at least one player quits,
and by $S^*$ the set of players who quit at stage $t^*$.
If no player ever quits, then $t^* = \infty$ and $S^* = \emptyset$.
The payoff to the players is $r^{S^*}$.

For convenience, whenever $S = \{i\}$ contains one element we write $r^i$ instead of $r^{\{i\}}$.
We will maintain the following assumption, which states that if a player quits alone, his payoff is 0.
This assumption is made without loss of generality, since adding a constant to the payoffs of a player
does not change his strategic considerations.

\begin{assumption}
\label{assumption:1}
For every $i \in I$ we have $r^i_i = 0$.
\end{assumption}

A quitting game is a strategic-form game in which the set of pure strategies of each player is $\dN \cup \{\infty\}$,
where the interpretation of the pure strategy $\infty$ is that the player never quits.
A (behavior) \emph{strategy} of player~$i$ is a sequence $x_i = (x_i^t)_{t \in \dN}$ of numbers in $[0,1]$,
with the interpretation that $x_i^t$ is the conditional probability that player~$i$ quits at stage $t$,
provided no player quit before that stage.
Denote by $X_i$ the set of all strategies of player~$i$,
by $X_{-i} := \times_{j \in I \setminus \{i\}} X_i$ the set of all strategy profiles of the other players,
and by $X := \times_{i \in I} X_i$ the set of \emph{strategy profiles}.

Every strategy profile $x = (x_i)_{i \in I} \in X$ induces a probability distribution $\prob_x$ over the set of plays.
We denote by $\E_x[\cdot]$ the corresponding expectation operator.
Denote by $\gamma(x) : = \E_x[r^{S^*}]$ the \emph{expected payoff} under strategy profile $x$.
A strategy profile $x$ is an \emph{$\ep$-equilibrium} if for every $i \in I$ and every strategy $x'_i \in X_i$ of player~$i$ we have
\[ \gamma_i(x) \geq \gamma_i(x'_i,x_{-i})-\ep. \]
Using the insights of Flesch, Thuijsman, and Vrieze (1997),
Solan (1999) proved that every three-player absorbing game, hence every three-player quitting game, admits an $\ep$-equilibrium for every $\ep>0$.
Solan and Vieille (2001) and Simon (2012) extended this result to multi-player quitting games that satisfy various conditions.

A strategy $x_i = (x_i^t)_{t \in \dN} \in X_i$ is \emph{stationary} if $x_i^t = x_i^{t'}$ for every $t,t' \in \dN$.
In this case we denote by $x_i$ the probability by which player~$i$ quits at every stage,
and we view a stationary strategy profile $x=(x_i)_{i \in I}$ as a vector in $[0,1]^N$.
We denote by $C_i$ (resp.~$Q_i$) the stationary strategy of player~$i$ in which he always continues (resp.~always quits).

The following observation asserts that if there is a normal player~$i$ such that $r^i \in \dR^N_{\geq 0}$,
then a stationary $\ep$-equilibrium exists.

\begin{lemma}
\label{lemma:stationary:eq}
If there is a normal player $i \in I_*$ such that $r^i \in \dR^N_{\geq 0}$,
then a stationary $\ep$-equilibrium exists, for every $\ep > 0$.
\end{lemma}

\begin{proof}
Since player~$i$ is normal, there is $j \in I$ such that $r^j_i \leq 0$.
The reader can verify that the following stationary strategy profile $x$
is a $4\ep$-equilibrium:
\[ x_i = \ep, \ \ \ x_j = \ep^2, \ \ \ x_k = C_k \ \ \forall k \in I \setminus \{i,j\}. \]
\end{proof}

\bigskip

Solan (1999) proved that in every three-player quitting game there is an $\ep$-equilibrium of one of two simple forms:
there is always a stationary $\ep$-equilibrium or
an $\ep$-equilibrium in which at every stage at most one player quits with positive probability.
Solan and Vieille (2002) provided a four-player quitting games
in which there is an $\ep$-equilibrium,
yet for $\ep > 0$ sufficiently small there is neither a stationary $\ep$-equilibrium nor
an $\ep$-equilibrium in which at every stage at most one player quits with positive probability.
To date it is not known whether four-player quitting games admit $\ep$-equilibria for every $\ep > 0$.

\subsection{Sunspot Equilibrium}

We enrich the game by introducing a public correlation device.
That is, at the beginning of every stage $t \in \dN$ the players observe a public signal $s^t \in [0,1]$ that is drawn by the uniform distribution,
independently of past signals.

A strategy of player~$i$ in the game with public correlation device is a sequence of measurable functions $\xi_i = (\xi_i^t)_{t \in \dN}$,
where $\xi^t_i : [0,1]^{t} \to [0,1]$.
The interpretation of $\xi_i^t$ is that if no player quits before stage $t$,
then at stage~$t$ player~$i$ quits with probability $\xi_i^t(s^1,s^2,\cdots,s^t)$.

Every strategy profile $\xi = (\xi_i)_{i \in I}$ induces a probability distribution over the set of plays in the game with public correlation device,
with a corresponding expectation operator that is denoted by $\E_\xi[\cdot]$.
Denote by $\gamma(\xi) := \E_\xi[r^{S_*}]$ the expected payoff under strategy profile $\xi$.

\begin{definition}
A strategy profile $\xi$ is a \emph{sunspot $\ep$-equilibrium} if it is an $\ep$-equilibrium in the game with public correlation device, that is, if
for every $i \in I$ and every strategy $\xi'_i$ of player~$i$ we have
\[ \gamma_i(\xi) \geq \gamma_i(\xi'_i,\xi_{-i})-\ep. \]
\end{definition}

The main result of this paper is the following.
\begin{theorem}
\label{theorem:main}
Every quitting game admits a sunspot $\ep$-equilibrium, for every $\ep > 0$.
\end{theorem}

\subsection{Normal and Abnormal Players}
\label{section:normal}

Simon (2012) defined the concepts of normal and abnormal players in quitting games.
According to Simon (2012), player~$i$ is \emph{normal} if there is a player $j \neq i$ such that $r^j_i \leq r^i_i$,
and he is \emph{abnormal} otherwise.
To prove our main result we could use these notions.
However, below we provide a condition that ensures that the game admits a stationary $\ep$-equilibrium,
for every $\ep > 0$; see Theorem~\ref{theorem:1}.
This condition can be proven using a weaker notion of normality, hence we elect to use this weaker notion and define it in this section.

Define inductively
\begin{eqnarray*}
I_0 &:=& I,\\
I_{l+1} &:=& \{i \in I_l \colon \hbox{ there exists } j \in I_l \hbox{ such that } r^j_i \leq 0\}.
\end{eqnarray*}
For each player $i \in I_{l+1}$ there is a player $j \in I_l$ who gives, by quitting alone, a nonpositive payoff to player~$i$.
Consequently,
\begin{equation}
\label{equ:abnormal}
i \in I_l, j \not\in I_l \ \ \ \Rightarrow \ \ \ r^i_j > 0.
\end{equation}
The sequence $(I_l)_{l \in \dN}$ is a decreasing sequence of sets, hence $I_* := \cap_{l \in \dN} I_l$ exists.
Eq.~(\ref{equ:abnormal}) implies that
\begin{equation}
\label{equ:abnormal1}
i \in I_*, j \not\in I_* \ \ \ \Rightarrow \ \ \ r^i_j > 0.
\end{equation}

\begin{definition}
Every player in the set $I_*$ is called \emph{normal},
and every player in the complement of $I_*$ is called \emph{abnormal}.
The number of normal players is denoted $n := |I_*|$.
\end{definition}

A player is normal according to Simon's (2012) definition if he belongs to the set $I_1$.
Thus, our definition of normality is a recursive application of Simon's definition.

Below we will construct sunspot $\ep$-equilibria in which at every stage at most one player quits.
Every player who quits alone receives 0, while by Eq.~\eqref{equ:abnormal1} every abnormal player receives a positive payoff when a normal player quits.
It is therefore not a wonder that in those equilibria we can ignore abnormal players:
if we can construct an equilibrium in which at every stage only one player quits, and this player is a normal player,
then the payoff of an abnormal player is positive,
hence he will not quit alone.

The following result allows us to assume that the set of normal players is nonempty.
Since there cannot be a single normal player, we deduce from this result that if an $\ep$-equilibrium does not exist for every $\ep > 0$,
then there are at least two normal players.
In fact, Theorem~\ref{theorem:1} below implies that if there are two normal players, then a stationary $\ep$-equilibrium exists.
\begin{lemma}
\label{lemma:25}
If all players are abnormal then there is a stationary $\ep$-equilibrium for every $\ep > 0$.
\end{lemma}

\begin{proof}
If $I_1 = \emptyset$ then $r^i \in \dR^N_{\geq 0}$ for every $i \in I$.
If $r^\emptyset \in \dR^N_{\geq 0}$, then the stationary strategy profile in which all players continue is a 0-equilibrium.
If there is a player $i \in I$ for which $r^\emptyset_i < 0$, then the stationary strategy profile
in which all players except player~$i$ continue, and player~$i$ quits at every stage with probability $\ep$,
is a $2\ep$-equilibrium, for every $\ep > 0$.

Suppose now that $I_1 \neq\emptyset$,
let $l \geq 1$ be the maximal index such that $I_l \neq \emptyset$,
and let $i \in I_l$ be arbitrary.
Since $i \in I_l \subseteq I_1$, there is some player $j \neq i$ such that $r^j_i \leq 0$.
Since $I_{l+1} = \emptyset$,
for every player $k \in I_l \setminus \{i\}$ we have $r^i_k > 0$.
By Eq.~(\ref{equ:abnormal}), for every player $k \not\in I_l$ we have $r^i_k > 0$.
It follows that $r^i \in \dR^n_{\geq 0}$, and therefore by Lemma~\ref{lemma:stationary:eq} a stationary $\ep$-equilibrium exists for every $\ep > 0$.
\end{proof}

\subsection{Linear Complementarity Problems and the Main Result}

Let $r^1,r^2,\cdots,r^n$ be $n$ vectors in $\dR^n$, and let $q \in \dR^n$.
The \emph{linear complementarity problem} $\lcp((r^i)_{i=1}^n,q)$
is the following problem that consists of linear equalities and inequalities:
\begin{eqnarray}
\nonumber
\hbox{Find}&&w \in \dR^n_{\geq 0}, \hbox{ and } z = (z_0,z_1,\cdots,z_n) \in \Delta(\{0,1,\cdots,n\}),\label{lpc}\\
\hbox{such that}
&& w = z_0q + \sum_{i =1}^n z_i r^i,\\
&&z_i = 0 \hbox{ or } w_i = 0, \ \ \ \forall i \in [n].
\nonumber
\end{eqnarray}
The last condition in the problem~(\ref{lpc}) is the \emph{complementarity condition}.

We note that for $q \in \dR^n_{\geq 0}$ there is always at least one solution to the problem~(\ref{lpc}), namely, $z = (1,0,\cdots,0)$ and $w = q$.

\begin{remark}
Let $R$ be an $n \times n$ matrix, and let $q \in \dR^n$.
In the literature, the linear complementarity problem $\lcp(R,q)$
is the following problem that consists of linear equalities and inequalities:
\begin{eqnarray}
\hbox{Find}&&z,w \in \dR^n_{\geq 0},\nonumber\\
\hbox{such that}
&&w = q + Rz,\label{lpc1}\\
&&z_i = 0 \hbox{ or } w_i = 0, \ \ \ \forall i \in [n].
\nonumber
\end{eqnarray}
We call the problem~(\ref{lpc}) a linear complementarity problem since it is obtained from problem~(\ref{lpc1})
by multiplication by the positive real number $z_0$,
provided $z_0 > 0$.
Lemma~\ref{lemma:observation} below implies that in our application, when $z_0 = 0$ a stationary $\ep$-equilibrium exists,
hence in the cases in which stationary $\ep$-equilibria do not exist the two problems~(\ref{lpc}) and~(\ref{lpc1}) are equivalent.
\end{remark}

\begin{definition}
An $n\times n$ matrix $R$ is called a \emph{$Q$-matrix} if for every $q \in \dR^n$ the linear complementarity problem $\lcp(R,q)$ has at least one solution.
\end{definition}

The authors are not aware of any characterization of $Q$-matrices.
The following example illustrates the concept of $Q$-matrices.

\begin{example}
\label{example:1}
Let $R$ be a $3 \times 3$ matrix that has the following sign form:
\[ R = \left(
\begin{array}{ccc}
0 & + & - \\
- & 0 & + \\
+ & - & 0
\end{array}
\right) \]
Theorem 6.2.7 in Berman and Plemmons (1994) implies that
the matrix $R$ is a $Q$-matrix if and only its determinant is positive.
\end{example}

Recall that the number of normal players is denoted by $n$.
Order the players so that the set of normal players is $I_* = [n]$.
For every normal player $i \in I_*$ we denote by $\widehat r^i \in \dR^n$ the restriction of the vector $r^i$ to the coordinates in $I_*$.
Denote by $\widehat R$ the $n \times n$ matrix whose $i$'th column is $\widehat r^i$.
By Assumption~\ref{assumption:1}, all entries on the diagonal of $\widehat R$ are 0.

Denote by $\vec 0$ the vector all of whose coordinates are 0.
The following observation states that if the linear complementarity problem $\lcp(\widehat R,\vec 0)$ has a nontrivial solution, then the game possesses
a stationary $\ep$-equilibrium for every $\ep > 0$.

\begin{lemma}
\label{lemma:observation}
If there is a solution $(w,z)$ of the linear complementarity problem $\lcp(\widehat R,\vec 0)$ that satisfies $z_0 < 1$,
then there is a stationary $\ep$-equilibrium for every $\ep > 0$.
\end{lemma}

We note that If the linear complementarity problem $\lcp(\widehat R,q)$ has a nontrivial solution $(w,z)$ with $z_0 = 0$,
then the problem $\lcp(\widehat R,\vec 0)$ has a nontrivial solution as well.

\bigskip

\begin{proof}
Let $(w,z)$ be a solution of the linear complementarity problem $\lcp(\widehat R,\vec 0)$ that satisfies $z_0 < 1$
and let $J := \{ i \in I_* \colon z_i > 0\}$.
Then $w_i = 0$ for every $i \in J$ and
$w = \sum_{i \in J} z_i r^i$.

If $J=\{i\}$ contains a single player, then necessarily $r^i = w \in \dR^N_{\geq 0}$,
and by Lemma~\ref{lemma:stationary:eq} there is a stationary $\ep$-equilibrium, for every $\ep > 0$.
Assume then that $|J| \geq 2$.

Fix $\ep > 0$ sufficiently small and consider the following stationary strategy profile $x$:
each player $i$ quits at every stage with probability $\ep z_i$.
Then $\left\|\gamma(x) - \sum_{i \in I_*} \frac{z_i}{1-z_0}r_i\right\|_\infty < 2\ep$,
so that $\left|\gamma_i(x) - \frac{w_i}{1-z_0}\right| < 2\ep$ for every $i \in I_*$,
and $\gamma_i(x) > -2\ep$ for every $i \not\in I_*$.

Since each player $i \in I_*$ quits with probability $\ep z_i$, it follows that no player can profit much by quitting:
\[ \gamma_i(Q_i,x_{-i}) \leq 2\ep \leq \gamma_i(x) + 4\ep. \]
If some player $i \in I_*$ with $z_i > 0$ deviates by always continuing, then his payoff $\gamma_i(C_i,x_{-i})$ is close to
$\sum_{j \in I_* \setminus \{i\}} \frac{z_j}{1-z_0-z_i} r^j_i$.
Note that since $|J| \geq 2$ the denominator does not vanish.
Since $z_i > 0$ we have $w_i = 0$, hence
$\sum_{j \in I_* \setminus \{i\}} \frac{z_j}{1-z_0-z_i} r^j_i = 0 \leq \gamma_i(x) + 2\ep$,
so the profit by this deviation is not large either.
The result follows.
\end{proof}

\bigskip

Lemma~\ref{lemma:25} and Lemma~\ref{lemma:observation}
show that for every $\ep > 0$ a stationary $\ep$-equilibrium exists
as soon as $I_* = \emptyset$,
or $I_* \neq \emptyset$ and the linear complementarity problem $\lcp(\widehat R,\vec 0)$ has a nontrivial solution.
The following theorem, which is proven in the next section,
completes the proof of Theorem~\ref{theorem:main}.
In addition to proving that in every quitting game a sunspot $\ep$-equilibrium exists,
it links the property of $\widehat R$ being a $Q$-matrix to the structure of the sunspot $\ep$-equilibria in the quitting game.

\begin{theorem}
\label{theorem:1}
Suppose that $I_* \neq \emptyset$ and
the linear complementarity problem $\lcp(\widehat R,\vec 0)$ does not have a nontrivial solution.
\begin{enumerate}
\item
If the matrix $\widehat R$ is not a $Q$-matrix, then the quitting game $\Gamma((r^S)_{S \subseteq I})$
has a stationary $\ep$-equilibrium, for every $\ep > 0$.
\item
If the matrix $\widehat R$ is a $Q$-matrix,
then for every $\ep > 0$ the quitting game $\Gamma((r^S)_{S \subseteq I})$ has a sunspot $\ep$-equilibrium
in which at every stage at most one player quits with positive probability.
\end{enumerate}
\end{theorem}

\subsection{An Example}

To illustrate the solution concept and our approach, in this section we provide the construction of a sunspot $\ep$-equilibrium
that uses only unilateral quittings in a specific game.
We will provide two constructions;
the first will be used in Section~\ref{section:characterization}
to characterize the set of sunspot equilibrium payoffs in a certain class of quitting games.
Unfortunately we could not generalize it to all quitting games.
The second construction will serve us in the proof of the general case.

Consider a quitting game with four players,
where the vectors $(r^i)_{i=1}^4$ are given by%
\footnote{In this example only we deviate from the assumption that payoffs are bounded by 1.}
\[ r^1 = (0,4,-1,-1), \ \ r^2 = (4,0,-1,-1), \ \ r^3 = (-1,-1,0,4), \ \ r^4 = (-1,-1,4,0). \]
The rest of the payoff function, namely, the vectors $r^\emptyset$ and $(r^S)_{|S| \geq 2}$, will not affect the analysis hence is omitted.
We note that all players are normal,
and that these payoffs are essentially the same payoffs that where used by Solan and Vieille (2002)
to construct a quitting game in which there is neither a stationary $\ep$-equilibrium nor
an $\ep$-equilibrium in which at every stage at most one player quits with positive probability.

\subsubsection{First Construction}
\label{section:example}

Observe that
\begin{eqnarray}
\label{equ:71}
(1,1,0,0) &=& \tfrac{1}{2}(2,0,0,0) + \tfrac{1}{2}(0,2,0,0)\\
&=& \tfrac{1}{2} \left( \tfrac{1}{2}(4,0,-1,-1) + \tfrac{1}{2}(0,0,1,1) \right)
+ \tfrac{1}{2} \left( \tfrac{1}{2}(0,4,-1,-1) + \tfrac{1}{2}(0,0,1,1)) \right),
\nonumber
\end{eqnarray}
and similarly
\begin{eqnarray}
\label{equ:72}
(0,0,1,1) &=& \tfrac{1}{2}(0,0,2,0) + \tfrac{1}{2}(0,0,0,2)\\
&=& \tfrac{1}{2} \left( \tfrac{1}{2}(-1,-1,4,0) + \tfrac{1}{2}(1,1,0,0) \right)
+ \tfrac{1}{2} \left( \tfrac{1}{2}(-1,-1,0,4) + \tfrac{1}{2}(1,1,0,0)) \right).
\nonumber
\end{eqnarray}
Fix $\ep > 0$ such that $\tfrac{1}{\ep}$ is an integer.
The following construction, in which the players implement the payoff vector $(1,1,0,0)$ as a sunspot equilibrium payoff, suggests itself:
\begin{itemize}
\item
Nature chooses whether the players implement the vector $(2,0,0,0)$ (if the current signal is smaller than $\tfrac{1}{2}$)
or the vector $(0,2,0,0)$ (if the current signal is at least $\tfrac{1}{2}$).
\item
If Nature chose to implement the vector $(2,0,0,0)$,
in each one of the next $\tfrac{1}{\ep}$ stages Player~2 quits with probability $\lambda$,
where $(1-\lambda)^{1/\ep} = \tfrac{1}{2}$.
That is, in each of these stages Player~2 quits with a small probability, and during these stages the total probability that he quits is $\tfrac{1}{2}$.
\item
If Nature chose to implement the vector $(0,2,0,0)$,
in each one of the next $\tfrac{1}{\ep}$ stages Player~1 quits with probability $\lambda$,
where $(1-\lambda)^{1/\ep} = \tfrac{1}{2}$.
\item
At the end of the $\tfrac{1}{\ep}$ stages, if no player quits,
the players turn to implement the vector $(0,0,1,1)$ in an analogous way.
\end{itemize}
We denote by $\xi^*$ the strategy profile that was just defined.
Under $\xi^*$ the game terminates with probability 1.
Moreover, even if one of the players deviates, the game terminates with probability 1.
Eqs.~\eqref{equ:71} and~\eqref{equ:72} imply that $\gamma(\xi^*) = (1,1,0,0)$,
and, more generally, that when the players attempt to implement a certain vector, say $(2,0,0,0)$, their expected payoff is that vector.

We now argue that no player can profit much by deviating from $\xi^*$.
To this end we note that the expected continuation payoff of all players after every history is nonnegative.
In each stage in which a player is supposed to quit with positive probability, his continuation payoff is 0,
hence at such stages the player is indifferent between continuing and quitting.

A player who is supposed to quit with a positive probability at a given stage,
does so with probability $\lambda$, which is small.
Consequently, since payoffs are bounded by 4 and by
Assumption~\ref{assumption:1}, if a player deviates and quits at a stage in which he is supposed to continue
his payoff is at most $4\lambda$.
Since the continuation payoff of all players after every history is nonnegative,
this implies that no player who is supposed to continue at a given stage can profit more than $4\lambda$ by quitting at that stage.

As we will see in Section~\ref{section:characterization},
this construction can be generalized to the case in which the matrix $\widehat R$ contains exactly one positive entry in each row and each column.

\subsubsection{Second Construction}

Fix $\ep \in (0,1)$.
We will use the following identities:%
\footnote{The verification of the calculations in this construction can be simplified by observing that the total sum of the coordinates
of all vectors in the construction is 2.}
\begin{equation}
\label{equ:73}
(0,0,2,0) = \tfrac{\ep}{6+\ep} (0,0,0,2) + \tfrac{6}{6+\ep}(0,0,\tfrac{6+\ep}{3},-\tfrac{\ep}{3}),
\end{equation}
and
\begin{eqnarray}
\nonumber
(0,0,\tfrac{6+\ep}{3},-\tfrac{\ep}{3})
&=& \tfrac{1}{6}(0,4\ep,2-3\ep,-\ep) +  \tfrac{1}{6}(4\ep,0,2-3\ep,-\ep) + \tfrac{4}{6}(-\ep,-\ep,2+2\ep,0)\\
\label{equ:74}
&=& \tfrac{1}{6} \bigl((1-\ep) (0,0,2,0) + \ep(0,4,-1,-1)\bigr)\\
\nonumber
&&+ \tfrac{1}{6} \bigl((1-\ep) (0,0,2,0) + \ep(4,0,-1,-1)\bigr)\\
&&+ \tfrac{4}{6} \bigl((1-\ep) (0,0,2,0) + \ep(-1,-1,4,0)\bigr).
\nonumber
\end{eqnarray}
These equalities suggest the following construction of a sunspot $5\ep$-equilibrium $\xi^*$ with $\gamma(\xi^*) = (0,0,2,0)$.
\begin{itemize}
\item   Nature chooses whether to implement
the vector $(0,0,0,2)$ (with probability $\tfrac{\ep}{6+\ep}$),
the vector $(0,4\ep,2-3\ep,-\ep)$ (with probability $\tfrac{1}{6+\ep}$),
the vector $(4\ep,0,2-3\ep,-\ep)$ (with probability $\tfrac{1}{6+\ep}$),
or the vector $(-\ep,-\ep,2+2\ep,0)$ (with probability $\tfrac{4}{6+\ep}$).
\item   If Nature chose to implement the vector $(0,0,0,2)$, then the players repeat the analogous construction with the appropriate amendments.
\item
If Nature chose to implement the vector $(0,4\ep,2-3\ep,-\ep)$,
then Player~1 quits with probability $\ep$ and continues with probability $1-\ep$.
\item
If Nature chose to implement the vector $(4\ep,0,2-3\ep,-\ep)$,
then Player~2 quits with probability $\ep$ and continues with probability $1-\ep$.
\item
If Nature chose to implement the vector $(-\ep,-\ep,2+2\ep,0)$,
then Player~4 quits with probability $\ep$ and continues with probability $1-\ep$.
\item   If no player quit, then the players implement the payoff $(0,0,2,0)$ as indicated above.
\end{itemize}
As in the first construction, under the strategy profile $\xi^*$
the game terminates with probability 1,
hence by Eqs.~\eqref{equ:73} and~\eqref{equ:74} the expected payoff under $\xi^*$ after every finite history
is the payoff vector that the players implement beginning at that stage.

We now argue that no player can profit more than $5\ep$ by deviating from $\xi^*$.
Note that the play terminates with probability 1 even if a single player deviates.
Consequently, when a player is supposed to quit with positive probability,
he is indifferent between continuing and quitting.
Moreover, if some player, say player~$i$, is supposed to continue after some finite history,
then his continuation payoff is at least $-\ep$,
while, since the player who quits with positive probability does so with probability $\ep$,
by quitting player~$i$ will obtain at most $4\ep$.
Thus no player can profit more than $5\ep$ by deviating from $\xi^*$.

\section{Proof of Theorem~\ref{theorem:1}}
\label{sec:proof}

In this section we prove Theorem~\ref{theorem:1}.
We start by describing the discounted game,
which will be used in the proof of the first claim of the theorem.
This claim will be proven in Section~\ref{section:proof:1} and the second claim will be proven in Section~\ref{section:proof:2}.

\subsection{The Discounted Game}

In this section we consider the discounted game, in which the vector $r^\emptyset$ represents the stage payoff until the game terminates.
Formally, given a discount factor $\lambda \in [0,1)$,
the \emph{$\lambda$-discounted game} $\Gamma_\lambda((r^S)_{S \subseteq I})$ is the strategic-form game
$(I,(X_i)_{i \in I}, \gamma^\lambda)$,
where the set of players coincides with the set of players in the original quitting game $\Gamma((r^S)_{S \subseteq I})$,
the set of strategies of each player $i \in I$ is $X_i$, his set of behavior strategies in the original quitting game,
and the payoff function is given by
\begin{eqnarray*}
\gamma^\lambda(x) &:=&
\E_x\left[ (1-\lambda) \sum_{t=1}^\infty \lambda^{t-1} \left( \mathbf{1}_{\{t < t_*\}} r^\emptyset + \mathbf{1}_{\{t \geq t_*\}} r^{S_*}\right)\right]\\
&=& \E_x[(1-\lambda^{t_*-1}) r^\emptyset + \lambda^{t_*-1} r^{S_*}].
\end{eqnarray*}
When $x$ is a stationary strategy profile, we have
\begin{equation}
\label{equ:discounted}
\gamma^\lambda(x) =
\frac{\lambda \prod_{i \in I} (1-x_i) r^\emptyset + \sum_{\emptyset \neq S \subseteq I} \left(\prod_{i \in S} x_i \prod_{i \not\in S} (1-x_i) r^S\right)}
{\lambda \prod_{i \in I} (1-x_i) + \sum_{\emptyset \neq S \subseteq I} \left(\prod_{i \in S} x_i \prod_{i \not\in S} (1-x_i)\right)}.
\end{equation}

A strategy profile $x \in X$ is a \emph{$\lambda$-discounted equilibrium} if for every player $i \in I$ and every strategy $x'_i \in X_i$ of player~$i$ we have
$\gamma_i^\lambda(x) \geq \gamma_i^\lambda(x'_i,x_{-i})$.


By Fink (1964) or Takahashi (1964) the $\lambda$-discounted game admits a $\lambda$-discounted equilibrium in stationary strategies.
By Bewley and Kohlberg (1976) one can choose a semi-algebraic function $\lambda \mapsto x^\lambda$
that assigns a stationary discounted equilibrium to each discount factor.
In particular, we can assume w.l.o.g.~that the limit $x^0 := \lim_{\lambda \to 0} x^\lambda$ exists.
Moreover, we can assume that either $x^\lambda_i = 0$ for every $\lambda$ sufficiently close to 0
or $x^\lambda_i > 0$ for every $\lambda$ sufficiently close to 0.

\subsection{Stationary $\ep$-Equilibria}
\label{section:proof:1}

In this section we prove the first statement of Theorem~\ref{theorem:1}.
Let $x$ be a stationary strategy profile.
If $\sum_{i \in I} x_i = 0$, then $x = \vec 0$, and the game continue forever.
If $\sum_{i \in I} x_i > 0$ then at least one player quits at every period with positive probability,
and the game terminates a.s.
It is well known (see, e.g., Vrieze and Thuijsman (1989) or Solan (1999))
that if $(x^\lambda)_{\lambda > 0}$ is a sequence of stationary strategy profiles such that
$x^0 := \lim_{\lambda \to 0} x^\lambda$ exists and if $\sum_{i\in I} x^0_i > 0$, then
\begin{equation}
\label{equ:limit}
\lim_{\lambda \to 0} \gamma^\lambda(x^\lambda) = \gamma(x^0) = \widehat\gamma(x^0) = \lim_{\lambda \to 0} \widehat\gamma^\lambda(x^\lambda).
\end{equation}

Suppose now that the matrix $\widehat R$ is not a $Q$-matrix.
Then there is a vector $\widehat q \in \dR^{n}$ such that the linear complementarity problem $\lcp(\widehat R,\widehat q)$ does not have any solution.
In particular, $\widehat q \not\in \dR^{n}_{\geq 0}$.
Extend $\widehat q$ to a vector in $\dR^N$ by setting all coordinates that are not%
\footnote{The new coordinates can be set to any positive number, and not necessarily to 1.}
in $[n]$ to 1, and denote by $q$ the resulting vector.

Consider the auxiliary quitting game $\Gamma((r^S)_{\emptyset \neq S \subseteq I},q)$,
where $q$ is the payoff if no player ever quits,
and the $\lambda$-discounted version of this game.
To distinguish the payoff in the original game from the payoff in the auxiliary game, we denote the former by $\gamma(x)$ and $\gamma^\lambda(x)$,
and the latter by $\widehat \gamma(x)$ and $\widehat \gamma^\lambda(x)$.
For every discount factor $\lambda \in (0,1]$ let $x^\lambda$ be a stationary equilibrium of the auxiliary quitting game
and denote $x^0 := \lim_{\lambda \to 0} x^\lambda$.

\bigskip
\noindent\textbf{Case 1:} $x^0$ is absorbing and under $x^0$ at least two players quits with positive probability.

We will show that in this case $x^0$ is a stationary 0-equilibrium.
Since under $x^0$ at least two players quit with positive probability,
the play eventually terminates even if one player deviates.
By Eq.~\eqref{equ:limit}, since $x^\lambda$ is a $\lambda$-discounted equilibrium of the auxiliary game $\Gamma_\lambda((r^S)_{\emptyset \neq S \subseteq I},q)$,
and by Eq.~(\ref{equ:limit}) once again,
we deduce that for every player $i \in I$ we have
\begin{equation}
\label{equ:limit1}
\gamma_i(x^0) = \widehat \gamma_i(x^0) = \lim_{\lambda \to 0} \widehat \gamma_i^\lambda(x^\lambda) \geq
\lim_{\lambda \to 0} \widehat \gamma^\lambda_{i}(Q_{i},x^\lambda_{-i})
= \widehat \gamma_i(Q_i,x^0_{-i}) = \gamma_i(Q_i,x^0_{-i})
\end{equation}
and
\begin{equation}
\label{equ:limit2}
\gamma_i(x^0) = \widehat \gamma_i(x^0) = \lim_{\lambda \to 0} \widehat \gamma_i^\lambda(x^\lambda) \geq
\lim_{\lambda \to 0} \widehat \gamma^\lambda_{i}(C_{i},x^\lambda_{-i})
= \widehat \gamma_i(C_i,x^0_{-i}) = \gamma_i(C_i,x^0_{-i}).
\end{equation}
It follows that no player can profit by deviating in the original quitting game.

\bigskip
\noindent\textbf{Case 2:} $x^0$ is absorbing and under $x^0$ there is a single player who quits with positive probability.

We will prove that for every $\ep > 0$,
by supplementing $x^0$ with a threat strategy against a deviation of the unique player who quits with positive probability under $x^0$,
we can construct a stationary $\ep$-equilibrium.

We argue that $x^\lambda_k = 0$ for every abnormal player~$k$ and every $\lambda$ sufficiently close to 0.
Indeed, let $k \not\in I_*$ be an abnormal player satisfying $x^\lambda_k > 0$ for every $\lambda$ sufficiently small.
Since $x^\lambda$ is an equilibrium of the $\lambda$-discounted game $\Gamma_\lambda((r^S)_{\emptyset \neq S \subseteq I},q)$
and by Eq.~\eqref{equ:limit},
\begin{equation}
\label{equ:201}
\lim_{\lambda \to 0} \widehat \gamma^\lambda_{k}(C_{k},x^\lambda_{-k}) \leq \lim_{\lambda \to 0} \widehat \gamma^\lambda_{k}(Q_{k},x^\lambda_{-k})
= \widehat \gamma_k(Q_k,x^0_{-k}) = 0.
\end{equation}
Under the strategy profile $(C_k,x^0_{-k})$ at most one player quits with positive probability,
hence the quantity $\lim_{\lambda \to 0} \widehat \gamma^\lambda_{k}(C_{k},x^\lambda_{-k})$
is a convex combination of $q_k$ and $(r^j_k)_{j \neq k}$.
Indeed, in Eq.~\eqref{equ:discounted}, when substituting $x$ by $x^\lambda$ and taking the limit as $\lambda$ goes to 0,
all summands that correspond to sets of players containing at least two players vanish.
Since player~$k$ is abnormal, we have $q_k =1$,
hence there is a player $j$ with $r^j_k \leq 0$.
This implies that $k \in I_1$, and inductively it implies that $k \in I_l$ for every $l \in \dN$, contradicting the assumption that $k$ is an abnormal player.

Denote by $i$ the unique player who quits with positive probability under $x^0$.
The discussion in the previous paragraph implies that player~$i$ is normal.
A possible deviation of player $j \neq i$ is to quit at some stage.
As in Eq.~(\ref{equ:limit1}), such a deviation is not profitable for player $j$.

A possible deviation of player~$i$ is to continue forever.
In this case his payoff will be $q_i$ rather than 0, so this deviation is possibly profitable.
Since $i$ is a normal player, there is a player $j \in I_*$, distinct from $i$,
such that $r^j_i \leq 0$.
It follows that the following strategy profile is a stationary $4\ep$-equilibrium:
at every stage,
player~$i$ quits with probability $\ep$,
player~$j$ quits with probability $\ep^2$,
and all other players continue.

\bigskip
\noindent\textbf{Case 3:} $x^0$ is nonabsorbing.

Since $x^\lambda$ is a $\lambda$-discounted equilibrium and by Eq.~(\ref{equ:limit}), we have
\begin{equation}
\label{equ:3:1}
\lim_{\lambda \to 0} \widehat \gamma_i^\lambda(x^\lambda) \geq \lim_{\lambda \to 0} \widehat \gamma^\lambda_{i}(Q_{i},x^\lambda_{-i})
= \widehat \gamma_i(Q_i,C_{-i}) = 0.
\end{equation}
As in Case~2, if $x^\lambda_i > 0$ for every $\lambda$ sufficiently small, then $i$ is a normal player.
Denote
\[ z_i^\lambda := \frac{x_i^\lambda}{\lambda+\sum_{j \in I_*} x_j^\lambda}, \ \ \ i \in I, \]
and
\[ z^\lambda_0 := \frac{\lambda}{\lambda+\sum_{j \in I_*} x_j^\lambda}. \]
Since $x^0 = \vec 0$ and since $x^\lambda_i = 0$ for every $\lambda$ sufficiently small and every abnormal player $i \not\in I_*$, we have by Eq.~\eqref{equ:limit}
\[ w := \lim_{\lambda \to 0} \widehat    \gamma^\lambda(x^\lambda) = \lim_{\lambda \to 0} \left( z^\lambda_0 q + \sum_{i \in I_*} z^\lambda_i r^i \right). \]
Set $z^0 := \lim_{\lambda \to 0} z^\lambda$.

Let $\widehat w$ and $\widehat z$ be the restriction of $w$ and $z^0$ to the first $n$ coordinates.
Note that $\widehat z$ is a probability distribution over $\{0,1,\cdots,n\}$.
We verify that $(\widehat w, \widehat z)$ is a solution of the linear complementarity problem $\lcp(\widehat R,\widehat q)$,
contradicting the assumption that this problem has no solution.
Indeed, by definition
\[ \widehat w = \widehat z_0 q + \sum_{i=1}^n \widehat z_i r^i. \]
Eq.~(\ref{equ:3:1}) implies that $\widehat w \in \dR^{n}_{\geq 0}$.
If $\widehat z_i > 0$ then $z_i^0 > 0$, and therefore $z^\lambda_i > 0$ for every $\lambda$ sufficiently close to 0,
hence $x^\lambda_i > 0$ for every $\lambda$ sufficiently close to 0.
This implies that player~$i$ is indifferent between continuing and quitting, so that by Eq.~(\ref{equ:limit}),
\begin{equation*}
\widehat w_i = \lim_{\lambda \to 0} \widehat \gamma_i^\lambda(x^\lambda) =
\lim_{\lambda \to 0} \widehat \gamma_{i}(Q_{i},x^\lambda_{-i}) = \widehat \gamma_{i}(Q_{i},C_{-i}) = 0.
\end{equation*}
The claim follows.

\subsection{Sunspot Equilibria In Which At Most One Player Quits At Every Stage}
\label{section:proof:2}

In this section we prove the second statement of Theorem~\ref{theorem:1}.

\subsubsection{The Set of Possible Sunspot Equilibrium Payoffs}

Our goal is to construct a sunspot $\ep$-equilibrium in which only normal players quit, at every stage at most one player quits,
and he does so with a low probability.
This has two consequences.
First, the equilibrium payoff will be in $\conv(r^1,\cdots,r^n)$.
Second, since $r^i_i = 0$ for every player $i \in I$, a player who deviates and quits when he should not, receives an amount close to 0.
Hence the equilibrium payoff will be close to the nonnegative orthant.
By Eq.~\eqref{equ:abnormal1} all abnormal players receive a nonnegative payoff when normal payoffs quit,
and since $r^i_i = 0$ for every player $i \in [N]$, the abnormal players are content having only normal players quit.
Since such a sunspot $\ep$-equilibrium depends only on the projection of the vectors $(r^i)_{i \in I}$ to normal players,
we will consider in this section the $n \times n$ matrix $\widehat R$ whose $i$'th column coincides with the vector $\widehat r^i$.
The set of all vectors that may be sunspot equilibrium payoff when only normal players quit, at each stage at most one player quits,
and he does so with small probability, is, then,
\[ D := \conv( \widehat r^1,\cdots,\widehat r^n) \cap \dR^n_{\geq 0}. \]

The following observation states that whenever $\widehat q$ is in the convex hull of $\{\widehat r^1,\cdots,\widehat r^n\}$,
then any vector $w$ that is part of a solution of the linear complementarity problem $\lcp(\widehat R,q)$ lies on the boundary of $D$.
\begin{lemma}
\label{lemma:boundary}
If $\widehat q \in \conv(\widehat r^1,\cdots,\widehat r^n)$ then every solution $(w,z)$ of $\lcp(\widehat R,\widehat q)$ satisfies $w\in \partial D$.
\end{lemma}

\begin{proof}
Fix a solution $(w,z)$ of $\lcp(\widehat R,\widehat q)$.
If the solution $(w,z)$ satisfies $z_0=1$, then $w=\widehat q$ and the result holds trivially.

Suppose then that $z_0 < 1$.
Since $\widehat q \in \conv(\widehat r^1,\cdots,\widehat r^n)$ we have $w \in \conv(\widehat r^1,\cdots,\widehat r^n)$.
Since $w \in \dR^n_{\geq 0}$ it follows that $w \in D$.
Since $z_0 < 1$ there is a player $i \in I$ such that $z_i > 0$, hence by the complementarity condition $w_i = 0$,
and therefore $w \in \partial D$.
\end{proof}

\subsubsection{The Basic Building Block}
\label{sec:theorem}

As mentioned before,
when $y \in \dR^n_{\geq 0}$ one solution $(w,z)$ of the problem $\lcp(R,y)$ is the trivial solution in which $w=y$ and $z=(1,0,0,\cdots,0)$.
The following theorem asserts that for every $y \in \partial D$ a nontrivial solution to a certain system that is related to problem~(\ref{lpc}) exists.
This theorem is the basic building block of our construction of a sunspot $\ep$-equilibrium.

\begin{theorem}
\label{theorem:2}
Under the assumptions of Theorem~\ref{theorem:1}(2),
for every $y \in \partial D$ and every $\ep > 0$ there are $w \in \partial D$, $w^1,\cdots,w^n\in \dR^n$, and $z \in \Delta(\{0,1,2,\ldots,n\})$
that satisfy the following conditions:
\begin{itemize}
\item[(F.1)]    $w^i \in \conv(w,\widehat r^i)\setminus \{w\}$ for every $i \in [n]$.
\item[(F.2)]    $w^i_j \geq -\ep$ for every $i,j \in [n]$.
\item[(F.3)]   $w = z_0y + \sum_{i=1}^n z_i w^i$.
\item[(F.4)]   If $i \in [n]$ and $z_i > 0$, then $w^i_i = 0$.
\item[(F.5)]   $\sum_{i=1}^n z_i > 0$.
\end{itemize}
\end{theorem}

Conditions~(F.1) and~(F.2) state that each $w^i$ is in the convex hull of $w$ and $\widehat r^i$, and each of its coordinates is at least $-\ep$.
Conditions~(F.3) and~(F.4) state that $(w,z)$ is a solution of the linear complementarity problem $\lcp(R,y)$,
when $(w^i)_{i=1}^n$ replace $(\widehat r^i)_{i=1}^n$,
and when the complementarity condition involves the vectors $(w^i)_{i=1}^n$ instead of the vector $w$.
Condition~(F.5) states that the solution is not trivial.

We note that the assumptions of Theorem~\ref{theorem:1}(2) and Condition~(F.1) imply that the unique
solution $\lambda_i \in (0,1]$ to the equation $w^i = \lambda_i \widehat r^i + (1-\lambda_i)w$ satisfies $\lambda_i < 1$,
provided $\ep$ is sufficiently small.
Indeed, if $\lambda_i = 1$ then $w = \widehat r^i$.
If $\ep$ is sufficiently small this implies that $\widehat r^i \in \dR^n_{\geq 0}$.
But then the linear complementarity problem $\lcp(\widehat R,\vec 0)$ has a nontrivial solution with $w=\widehat r^i$,
a contradiction to the assumptions of Theorem~\ref{theorem:1}(2).

\bigskip

\centerline{\includegraphics{figure.6}}

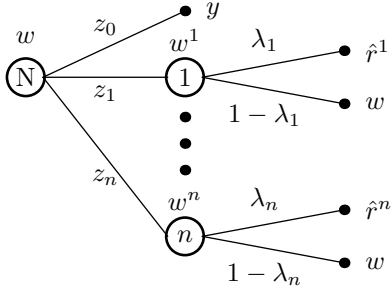
\centerline{Figure \arabic{figurecounter}: A graphic depiction of Theorem~\ref{theorem:2}.}

\bigskip

Figure~\arabic{figurecounter} provides a graphical interpretation to Theorem~\ref{theorem:2}.
Nature chooses an element $i \in \{0,1,\cdots,n\}$ according the distribution $z$.
If Nature chooses 0, the outcome is $y$.
If Nature chooses $i \in [n]$, then player~$i$ quits with probability $\lambda_i$.
If player~$i$ quits, the outcome is $\widehat r^i$, and otherwise it is $w$.

Condition~(F.4) asserts that every player who may be chosen is indifferent between quitting and continuing,
Condition~(F.1) asserts that the expected outcome if player~$i$ is chosen is $w^i$.
This condition moreover implies that $\lambda_i > 0$ for every $i \in [n]$,
so that every player who is chosen, quits with positive probability.
Condition~(F.3) asserts that the expected outcome at the root is $w$.
By Condition~(F.5) we have $\sum_{i=1}^n z_i > 0$, hence some player $i$ quits with positive probability.

Figure~\arabic{figurecounter} can describe the behavior of the players in a single stage of the quitting game:
Nature's signal chooses an element of $\{0,1,\cdots,n\}$ according to the distribution $z$.
If the choice is 0, no player quits;
if the choice is $i$, player~$i$ quits with probability $\lambda_i$,
while all other players continue.
We will use a proper concatenation of this behavior to construct a sunspot $\ep$-equilibrium in the quitting game $\Gamma((r^S)_{S \subseteq I})$.

In the above interpretation, if player~$i$ is chosen by nature, he quits with probability $\lambda_i$.
In quitting games players can quit simultaneously, and thus, if player~$j$ quits when player~$i$ is chosen,
the expected outcome will be $\lambda_i \widehat r^{\{i,j\}} + (1-\lambda_i) \widehat r^{j}$.
Since player~$j$'s payoff in this case, $\lambda_i \widehat r^{\{i,j\}}_j + (1-\lambda_i) \widehat r^{j}_j = \lambda_i \widehat r^{\{i,j\}}_j$,
may be higher than $w^i_j$, which is his expected outcome given that player~$i$ is chosen, player~$j$ may find it beneficial to quit when player~$i$ is chosen.
As in the example in Section~\ref{section:example},
to ensure that this type of deviation is not profitable, when player~$i$ is chosen, he will not quit in a single stage of the quitting game,
but rather along a block of $K$ stages, where $K$ is sufficiently large;
that is, in each stage of the block, he will quit with probability $1-(1-\lambda_i)^{1/K}$.
The expected continuation payoff along the block will thus be in the convex hull of $w^i$ and $w$.
Since $w \in \partial D\subset \dR^n_{\geq 0}$ and $w^i_j \geq -\ep$ for every $j \in [n]$,
this implies that the expected continuation payoff for all players along the block is at least $-\ep$,
so that a player who is supposed to continue throughout the block
cannot profit much by deviating and quitting.

Since $w$ is both an outcome of the tree that appears in Figure~\arabic{figurecounter}
and the expected outcome of this interaction,
\addtocounter{figurecounter}{1}
we can create a repeated version of this game, in which, if one of the players is chosen and this player does not quit,
then another copy of the game is played, see Figure~\arabic{figurecounter}.
Since $z_0 + \sum_{i=1}^n \lambda_iz_i > 0$,
the length of a play in the tree that appears in Figure~\arabic{figurecounter} is distributed according to a geometric distribution.
We call this auxiliary game $G(y)$.
Note that the possible outcomes of $G(y)$ are $\widehat r^1,\widehat r^2,\cdots,\widehat r^n,y$,
and the payoff under the behavior described above is $w$.

Denote by $w(y)$ the vector $w$ that corresponds to $y \in \partial D$ in Theorem~\ref{theorem:2}.
The natural approach to construct a sunspot $\ep$-equilibrium in the original quitting game
would be to find a sequence $(y^k)_{k \in \dN}$ such that $y^{k+1} = w(y^k)$,
and to concatenate the games that appear in Figure~\arabic{figurecounter} one after the other.
This is the approach that we take, though it requires some significant amendments.

\bigskip

\centerline{\includegraphics{figure.7}}

\centerline{Figure \arabic{figurecounter}: The auxiliary game $G(y)$ with geometric length.}
\addtocounter{figurecounter}{1}

\bigskip

By Theorem~\ref{theorem:2} we can choose for every $y \in \partial D$ a point $w(y) \in \partial D$, points $(w^i(y))_{i \in J(y)} \subset \dR^n$,
and a probability distribution $z(y) \in \Delta(\{0,1,\cdots,n\})$ that satisfy Conditions (F.1)--(F.5).
As described in Section~\ref{sec:theorem},
these quantities reflect an $\ep$-equilibrium behavior in an auxiliary quitting game with geometric length:
$w(y)$ is a sunspot $\ep$-equilibrium payoff in the game with continuation payoff $y$.
To construct a sunspot $\ep$-equilibrium in the quitting game,
we would like to concatenate such $\ep$-equilibria.

If $y$ were a sunspot $\ep$-equilibrium in the game with continuation payoff $w(y)$, this could be done as follows:
we would choose an arbitrary $y^0 \in \partial D$ and define inductively $y^{k+1} := w(y^k)$.
We would then implement a sunspot $\ep$-equilibrium in the original game by playing first the $\ep$-equilibrium that corresponds to the payoff $y^0$
in the auxiliary game with geometric length $G(y^1)$ with continuation payoff $y^1$,
then the $\ep$-equilibrium that corresponds to the payoff $y^1$
in the auxiliary game with geometric length $G(y^2)$ with continuation payoff $y^2$, and so on.

As soon as the total probability of termination under this construction is 1,
the resulting strategy profile would be a sunspot $\ep$-equilibrium in the original quitting game.
There are two problems in implementing this approach.

First, $w(y)$ is a sunspot $\ep$-equilibrium payoff in the auxiliary game with continuation payoff $y$, and not vice versa.
Hence, if we choose $y^0 \in \partial D$ arbitrarily and define inductively $y^{k+1} := w(y^k)$,
then time goes backwards:
we should choose some large $K \in \dN$,
play first a sunspot $\ep$-equilibrium with payoff $y^K$ in the auxiliary game with geometric length $G(y^{K-1})$ with continuation payoff $y^{K-1}$,
then a sunspot $\ep$-equilibrium with payoff $y^{K-1}$ in the auxiliary game with geometric length $G(y^{K-2})$ with continuation payoff $y^{K-2}$,
and so on, until we play a sunspot $\ep$-equilibrium with payoff $y^1$ in the auxiliary game with geometric length $G(y^0)$ with continuation payoff $y^{0}$.
After that we let the player play in an arbitrary way.
If the probability that the game is not terminated by a player before we end playing the sequence of sunspot $\ep$-equilibria is small,
then the way players play after implementing the sunspot $\ep$-equilibrium in $G(y^0)$ does not affect much the payoff,
and we would still obtain a sunspot approximate equilibrium.

The second issue concerns the probability of termination.
The probability that the play in the auxiliary game with geometric length $G(y^k)$ with continuation payoff $y^{k}$
 terminates by a player under the sunspot $\ep$-equilibrium with payoff $y^{k+1}$
is $\sum_{i \in I_*} z_i(y^k)$.
By Condition~(F.5) this quantity is positive,
but we do not have a uniform lower bound on it.
Hence, we cannot ensure that the probability of termination by a player under the finite concatenation of sunspot $\ep$-equilibria in auxiliary games
with geometric length can be made arbitrarily high.
To overcome this difficulty we will construct a sequence $(y^k)_{k=1}^K$ that satisfies the required properties approximately.
This approach is close to Theorem~3 in Simon (2007).

\subsubsection{Proof of Theorem~\ref{theorem:2}}

To prove Theorem~\ref{theorem:2} we need two notations.
For every nonempty set $J \subseteq I$ of indices define
\[ S(J) := \conv\{\widehat r^i, i \in J\}. \]
For every $y \in \dR^n$ denote
\[ J_y := \{ i \in [n] \colon y_i = 0\}. \]

The proof of Theorem~\ref{theorem:2} is divided into three cases.
Let $y \in \partial D$.
\begin{itemize}
\item   The case $y \in S(J_y)$ holds trivially, since we can take $w^i=(1-\ep)y + \ep \widehat r^i$ for each $i \in [n]$,
$w = y$, and $(z_0,z_1,\cdots,z_n)$ is a probability distribution that satisfies that $y = \sum_{i \in J_y} z_i \widehat r^i$.
\item   The case $y \not\in S(J_y)$ and $\conv(S(J_y),y) \cap D = \{y\}$ is handled in Lemma~\ref{lemma:23}.
\item   The case $y \not\in S(J_y)$ and $\conv(S(J_y),y) \cap D \supsetneqq \{y\}$ is handled in Lemma~\ref{lemma:24}.
\end{itemize}

\begin{lemma}
\label{lemma:23}
Let $y \in \partial D$ such that $y \not\in S(J_y)$.
If $\conv(S(J_y),y) \cap D = \{y\}$ then the conclusion of Theorem~\ref{theorem:2} holds.
\end{lemma}

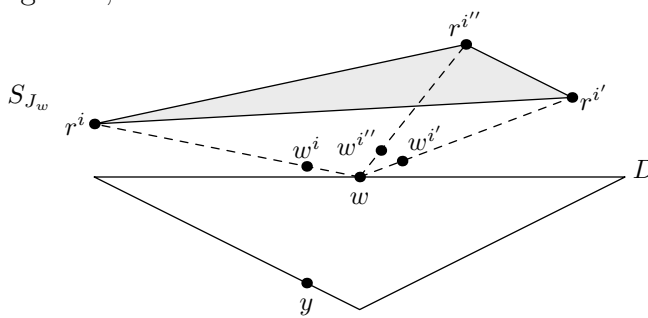
\bigskip

\centerline{\includegraphics{figure.2}}

\centerline{Figure \arabic{figurecounter}, Part A: The construction in Lemma~\ref{lemma:23}.}

\centerline{\includegraphics{figure.4}}

\centerline{Figure \arabic{figurecounter}, Part B: The construction in Lemma~\ref{lemma:23}.}
\bigskip

\begin{proof}
For every $\ep > 0$ let $y_\ep \in \conv(S(J_y),y) \setminus \{y\}$ satisfy $d(y_\ep,y) \leq \ep$ (see Figure~\arabic{figurecounter}(A)).
Since $y \in D \subseteq \conv(\widehat r^1,\cdots,\widehat r^n)$, it follows that $y_\ep \in \conv(\widehat r^1,\cdots,\widehat r^n)$.
Since $\conv(S(J_y),y) \cap D = \{y\}$ we have $y_\ep \not\in D = \conv(\widehat r^1,\cdots,\widehat r^n) \cap \dR^n_{\geq 0}$, and therefore
$y_\ep \not\in \dR^n_{\geq 0}$.

Let $(w_\ep,z_\ep) \in \dR^n_{\geq 0} \times \Delta(\{0,1,\cdots,n\})$ be a solution of the linear complementarity problem $\lcp(R,y_\ep)$,
so that $w_{\ep,i} = 0$ or $z_{\ep,i} = 0$ for every $i \in [n]$, and
\[ w_\ep = z_{\ep,0} y_\ep + \sum_{i \in I_*} z_{\ep,i}\widehat r^i. \]
By Lemma~\ref{lemma:boundary}, $w_\ep \in \partial D$.
In particular, $w_\ep \in\conv(S(J_{w_\ep}),y_\ep) \cap D$.
By taking a subsequence we can assume that the limits
\[
w := \lim_{\ep \to 0} w_\ep, \ \ \
z := \lim_{\ep \to 0} z_\ep, \ \ \
y := \lim_{\ep \to 0} y_\ep
\]
exist.
We can moreover assume that the sets $(J(w_\ep))_{\ep > 0}$ are independent of $\ep$.
Since $w_\ep \in \partial D$ for every $\ep > 0$ it follows that $w \in \partial D$.
Note that $w = z_0 y + \sum_{i \in I_*} z_{i} \widehat r^i$.
Furthermore, for every $i \in [n]$ we have $w_i = 0$ or $z_i = 0$.
We argue that $w \neq y$.

Indeed, assume by contradiction that $y = w = \lim_{\ep \to 0} w_\ep$.
It follows that $J_{w_\ep} \subseteq J_y$ for every $\ep > 0$ sufficiently small.
In particular, $S(J_{w_\ep}) \subseteq S(J_y)$ for every $\ep > 0$ sufficiently small.
Since $\conv(S(J_y),y) \cap D = \{y\}$ and since $y_\ep \in \conv(S(J_y),y) \setminus \{y\}$,
we conclude that $\conv(S(J_{w_\ep}),y_\ep) \cap D = \emptyset$. But $w_\ep \in\conv(S(J_{w_\ep}),y_\ep) \cap D$, a contradiction.

Define
\[ \widehat z_0 := \frac{\ep z_0}{\ep z_0 + \sum_{i=1}^n z_i}, \ \ \ \ \ \widehat z_i := \frac{z_i}{\ep z_0 + \sum_{i=1}^n z_i}, \ \ \ \forall i \in [n]. \]
Note that $z_i > 0$ if and only if $\widehat z_i > 0$.
Therefore $w_i = 0$ or $\widehat z_i = 0$ for every $i \in [n]$.
Define for every $i \in J_w$,
\begin{equation}
\label{equ:w1}
w^i := (1-\ep) w + \ep \widehat r^i,
\end{equation}
see Figure~\arabic{figurecounter}(B).

We argue that the conclusion of Theorem~\ref{theorem:2} holds for $w$, $(w^i)_{i \in [n]}$, and $(\widehat z^i)_{i \in [n]}$.
By construction Conditions~(F.1) and~(F.2) hold.
If $\widehat z_i > 0$ then $z_i > 0$, hence $w_i=0$, which implies that $w^i_i = 0$, so that Condition~(F.4) holds as well.
Since $w = z_0 y + \sum_{i \in I_*} z_{i} r^i$,
and since $w \neq y$,
it follows that $\sum_{i \in [n]} z_{i} > 0$,
and therefore $\sum_{i \in [n]} \widehat z_{i} > 0$, implying that Condition~(F.5) holds.
We now verify that Condition~(F.3) holds as well.
By Condition~(F.2) and Eq.~\eqref{equ:w1},
\begin{eqnarray*}
\ep w =  \ep z_0y + \sum_{i \in I_*}  \ep z_i \widehat r^i
= \ep z_0y + \sum_{i \in I_*}  z_i (w^i - (1-\ep) w).
\end{eqnarray*}
This implies that
\[ w = \frac{\ep z_0y + \sum_{i \in I_*}  z_i w^i}{\ep z_0 + \sum_{i \in I_*}  z_i}
= \widehat z_0 y + \sum_{i \in I_*} \widehat z_i w^i, \]
and Condition~(F.3) holds as well.
\end{proof}
\addtocounter{figurecounter}{1}

\begin{lemma}
\label{lemma:24}
Let $y \in \partial D$ such that $y \not\in S(J_y)$.
If $\conv(S(J_y),y) \cap D \supsetneqq \{y\}$
then the conclusion of Theorem~\ref{theorem:2} holds.
\end{lemma}

\bigskip

\centerline{\includegraphics{figure.1}}

\centerline{Figure \arabic{figurecounter}, Part A: The construction in Lemma~\ref{lemma:24}.}

\centerline{\includegraphics{figure.3}}

\centerline{Figure \arabic{figurecounter}, Part B: The construction in Lemma~\ref{lemma:24}.}
\addtocounter{figurecounter}{1}

\bigskip

\begin{proof}
Fix $\ep > 0$ sufficiently small.
Assume first that there is $i \in J_y$ such that $\conv(\widehat r^i, y) \cap D \supsetneqq \{y\}$, see Figure~\arabic{figurecounter}(A).
Since $r^i \not\in D$ while $y \in \partial D$ and $\conv(\widehat r^i, y) \cap D \supsetneqq \{y\}$,
it follows that there is $w \in \conv(\widehat r^i, y) \cap \partial D$.
In particular, there is $\lambda \in (0,1)$ such that $w = \lambda \widehat r^i + (1-\lambda)y$.
Since $i \in J(y)$ it follows that $w_i = 0$.
Thus, Theorem~\ref{theorem:2} holds with $w^i = (1-\ep)w + \ep \widehat r^i$ for every $i \in I_*$,
and $z$ that is defined by
\[ z_0 := \frac{\ep(1-\lambda)}{\ep(1-\lambda)+\lambda}, \ \ \ \ \
z_i := \frac{\lambda}{\ep(1-\lambda)+\lambda}. \]
Indeed, since $w = \lambda r^i + (1-\lambda)y$
we have
\[ \ep w = \ep\lambda \widehat r^i + \ep(1-\lambda)y
= \lambda(w^i=(1-\ep)w) + \ep(1-\lambda)y, \]
so that
\[ w = \frac{\ep(1-\lambda)}{\ep(1-\lambda)+\lambda}y + \frac{\lambda}{\ep(1-\lambda)+\lambda} \widehat r^i. \]

Assume now that $\conv(\widehat r^i, y) \cap D = \{y\}$ for every $i \in J_y$ and consider the set (see Figure \arabic{figurecounter}(B))
\[ \widehat S_\ep := (1-\delta) y + \delta S(J_y). \]
Since $\conv(\widehat r^i,y) \cap D = \{y\}$, the set $\widehat S_\ep$ is not a subset of $D$.
Moreover, provided $\ep$ is sufficiently small, this set intersects $D$.
In particular, there is a point $w \in \partial D \cap \widehat S_\ep$.
Since $\delta > 0$, we have $w \neq y$.

For every player $i \in J(y)$ define
\[ w^i := (1-\ep) y + \ep r^i. \]
The points $(w^i)_{i \in J(y)}$ are the extreme points of the set $\widehat S_\ep$,
and therefore there is a probability distribution $z \in \Delta(J_{y})$ such that $w = \sum_{i \in I_*} z_i w^i$.
The reader can verify that the conclusion of Theorem~\ref{theorem:2} holds with $w$, $(w^i)_{i \in J_y}$, and $z$.
\end{proof}

\subsubsection{An Approximation Result}

We start by a technical observation that will serve as an approximation tool.

\begin{theorem}
\label{theorem:approximate}
Let $(X,d)$ be a complete metric space and let $f : X \to X$ be a function that does not have any fixed point.
For every $c,C \in \dR_{\geq 0}$ there are $K \in \dN$ and a sequence $(x^k)_{k=1}^K$ of points in $X$ such that the following properties hold:
\begin{itemize}
\item[(A.1)]   $\sum_{k=1}^K d(x^k,f(x^k)) > C$.
\item[(A.2)]   $\sum_{k=1}^{K-1} d(x^{k+1},f(x^k)) < c$.
\end{itemize}
\end{theorem}

Figure~\arabic{figurecounter} provides a graphical depiction of Theorem~\ref{theorem:approximate};
each solid line represents the distance between some $x^k$ and $f(x^k)$,
and each dashed line represents the distance between some $f(x^k)$ and $x^{k+1}$.
The theorem claims that the total length of the solid lines is above $C$,
while the total length of the dashed lines is less than $c$.

\bigskip

\centerline{\includegraphics{figure.8}}

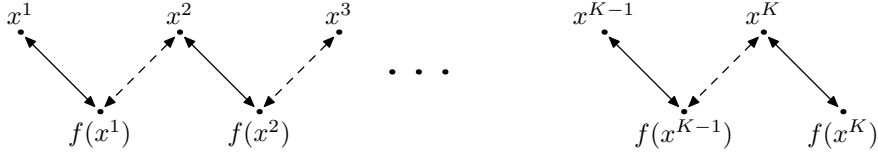
\centerline{Figure \arabic{figurecounter}: The construction in Theorem~\ref{theorem:approximate}.}
\addtocounter{figurecounter}{1}
\bigskip

\begin{proof}
The proof uses a transfinite construction.
We define an ordinal $\alpha_*$ and a sequence $(x^\alpha)_{\alpha < \alpha_*}$ as follows:
\begin{itemize}
\item[(TI.1)]   $x^0 \in X$ is arbitrary.
\item[(TI.2)]   If $\alpha$ is a successor ordinal set $x^{\alpha} := f(x^{\alpha-1})$.
\item[(TI.3)]   If $\alpha$ is a limit ordinal and $\sum_{\beta < \alpha} d(x^\beta,f(x^\beta)) = \infty$,
set $\alpha_* := \alpha$ and terminate the definition of the sequence.
\item[(TI.4)]   If $\alpha$ is a limit ordinal and $\sum_{\beta < \alpha} d(x^\beta,f(x^\beta)) < \infty$, set $x^\alpha := \lim_{\beta < \alpha} x^\beta$.
\end{itemize}
We note that if $\alpha$ is a limit ordinal and $\sum_{\beta < \alpha} d(x^\beta,f(x^\beta)) < \infty$
then for every $\ep > 0$ there is an ordinal $\alpha^\ep < \alpha$ for which $\sum_{\alpha^\ep \leq \beta < \alpha} d(x^\beta,f(x^\beta)) < \ep$,
which implies the existence of the limit $\lim_{\beta < \alpha} x^\beta$.
Since the space $X$ is complete, the limit $\lim_{\beta < \alpha} x^\beta$ in Case~(TI.4) is in $X$,
hence the definition of the sequence $(x^\alpha)_{\alpha < \alpha_*}$ is valid.

Since $f$ has no fixed point, $d(x,f(x)) > 0$ for every $x \in X$, and therefore the construction ends at some ordinal $\alpha_*$.
In fact, since the set of rational numbers is dense in the set of real numbers,
the ordinal $\alpha_*$ is a countable ordinal.
Since $\sum_{\beta < \alpha_*} d(x^\beta,f(x^\beta)) = \infty$,
there is an ordinal $\alpha_1$ such that $\sum_{\beta < \alpha_1} d(x^\beta,f(x^\beta)) > C+1$
By definition,
\[ \sum_{\alpha < \alpha_*} d(x^\alpha,f(x^\alpha)) = \sup_{A}\sum_{\alpha \in A} d(x^\alpha,f(x^\alpha)), \]
where $A$ ranges over all finite sets of ordinals smaller than $\alpha_*$,
hence there is a finite set $A$ of ordinals smaller than $\alpha_1$ such that the following two conditions hold:
\begin{itemize}
\item[(A.1')]   $\sum_{\alpha \in A} d(x^\alpha,f(x^\alpha)) > C$.
\item[(A.2')]   $\sum_{\alpha \not\in A, \alpha < \alpha_1} d(x^\alpha,f(x^\alpha)) < c$.
\end{itemize}
Denote $K := |A|$ and $A = \{u^1,u^2,\cdots,u^K\}$,
and assume that $u^1 < u^2 < \cdots < u^K$.
By Condition~(A.1'),
\[ \sum_{k=1}^K d(x^{u^k},f(x^{u^k})) = \sum_{\alpha \in A} d(x^\alpha,f(x^\alpha)) > C. \]
By the triangle inequality and Condition~(A.2'),
\begin{eqnarray*}
\sum_{k=1}^K d(x^{u^k},f(x^{u^{k+1}}))
&\leq&
\sum_{u^k \leq \alpha < u^{k+1}}d(f(x^\alpha),x^{\alpha+1})\\
&\leq&
\sum_{\alpha \not\in A, \alpha < \alpha_1} d(x^\alpha,f(x^\alpha)) < c.
\end{eqnarray*}
The result follows.
\end{proof}

%

\subsubsection{Constructing a Strategy Profile $\xi^*$}

Fix $\ep > 0$ sufficiently small so that each of the vectors $\widehat r^i$ contains an entry that is smaller than $-\ep$.
We now define a strategy profile $\xi^*$ in the quitting game,
which will turn out to be a sunspot $10\ep$-equilibrium.

We note that under the assumptions of Theorem~\ref{theorem:1}, the function $w : \partial D \to \partial D$ does not have a fixed point.
Indeed, the existence of such a fixed point implies that the linear complementarity problem $\lcp(\widehat R,\vec 0)$
has a nontrivial solution.

By Theorem~\ref{theorem:approximate} applied to $C = \tfrac{\binom{n}{2} \cdot 2(1+\ep)}{\ep^2}$, $c=\ep$,
$X = \partial D$ endowed with the supremum norm, and $f(y) = w(y)$ for every $y \in \partial D$,
there are $K \in \dN$ and a sequence $(y^k)_{k=1}^K$ that satisfy
\begin{itemize}
\item[(A.1'')]   $\sum_{k=1}^K \|y^k-w(y^k)\|_\infty > \tfrac{\binom{n}{2} \cdot 2(1+\ep)}{\ep^2}$.
\item[(A.2'')]   $\sum_{k=1}^{K-1} \|y^{k+1}-w(y^k)\|_\infty < \ep$.
\end{itemize}

For every $k \in [K]$ let $C_k \in \dN$ be sufficiently large such that
\[ 1-(1-\lambda_i(y^{k}))^{1/C_k} < \ep, \ \ \ \forall i \in I_*. \]
We argue that $\lambda_i(y^k) < 1$, hence such $C_k$ exists.
Indeed, if $\lambda_i(y^k) = 1$ then necessarily $\widehat r^i \in \dR^n_{\geq 0}$,
which implies that the linear complementarity problem $\lcp(\widehat R,\vec 0)$ has a nontrivial solution,
contradicting the assumptions.

The strategy profile $\xi^*$ that we will construct will yield a payoff close to $w(y^K)$.
We will partition the set of stages $\dN$ into $K+1$ kiloblocks of random (possibly infinite) size.
For each $k \in [K]$, kiloblock $k$ will mimic the sunspot $\ep$-equilibrium in the auxiliary game with geometric length $G(y^{K-k+1})$
with continuation payoff $y^{K-k+1}$
that yields equilibrium payoff $w(y^{K-k+1})$.
The last kiloblock will represent the rest of the game.
Condition~(A.1'') will imply that under $\xi^*$ with high probability the play terminates in one of the first $K$ kiloblocks.
Condition~(A.2'') will imply that the fact that the equilibrium payoff of the play in the $k$'th kiloblock, namely, $w(y^{K-k+1})$,
differs from the continuation payoff in the auxiliary game $G(y^{K-k})$,
does not affect much the payoffs of the players.

Partition the set of stages $\dN$ into $K+1$ kiloblocks of random (possibly infinite) size as follows.
For $k \in [K]$, the kiloblock is divided into blocks of size $C_{k}$; each block has a \emph{type} from the set $\{0,1,\cdots,n\}$.
At the beginning of the kiloblock, as well as at the end of each block, the type of the coming block is chosen by nature.
\begin{itemize}
\item   With probability $z_i$ the type of the next block is $i$.
\item   With probability $z_0$ the type of the next block is 0, this block is the last block of the kiloblock,
and the next kiloblock starts once this block ends.
\end{itemize}
The last kiloblock, which is the $(K+1)$'th kiloblock, is not divided into blocks and contains all remaining stages.

Define a strategy profile $\xi^*$ as follows. At stage $t$,
\begin{itemize}
\item   If $t$ lies in a block of type $i \in [n]$ is kiloblock $k \in [K]$,
then at stage $t$ player~$i$ quits with probability $1-(1-\lambda_i(y^{K-k+1}))^{1/C_{K-k+1}}$,
and all other players continue.
\item   If $t$ lies in the last block of a kiloblock (and then its type is necessarily 0),
or in the $(K+1)$'th kiloblock, then all players continue at stage $t$.
\end{itemize}

\subsection{The Strategy Profile $\xi^*$ is a Sunspot $7\ep$-equilibrium.}

In this section we prove that the strategy profile $\xi^*$ is a sunspot $7\ep$-equilibrium.
We will use the following inequality, which holds since,
by Condition~(F.3), $w-y = \sum_{i=1}^n z_i(w^i-y)$:
\begin{equation}
\label{equ:62}
\|w-y\|_\infty \leq 2\sum_{i=1}^n z_i.
\end{equation}

We first prove that the expected payoff of the normal players under the strategy profile $\xi^*$ is close to $w(y^K)$.
\begin{lemma}
\label{lemma:equilibrium}
For every normal player $i \in I_*$ we have
$| \gamma_i(\xi^*) - w_i(y^K)| < 2\ep$.
\end{lemma}

\begin{proof}
Fix a normal player $i \in I_*$.
Define a stochastic process $\eta_i = (\eta_i^k)_{k=1}^{K+1}$ as follows:
\begin{itemize}
\item   If the play was terminated before kiloblock $k$ by the set of players $S_*$, set
\begin{equation}
\label{equ:16.1}
\eta_i^k := \widehat r_i^{S_*} + \sum_{l < k} \| y_i^{K-l+1} - w_i(y^{K-l}) \|_\infty,
\end{equation}
where $S_*$ is the set of players who quit at stage $t_*$.
\item   If the play was not terminated before kiloblock $k$,
set
\begin{equation}
\label{equ:16.2}
\eta_i^k := w_i(y^{K-k+1}) + \sum_{l < k} \| y_i^{K-l+1} - w_i(y^{K-l}) \|_\infty. 
\end{equation}
\end{itemize}

By Eq.~\eqref{equ:62} and Condition~(A.1'') we have $\sum_{k=1}^K z_i(y^k) \geq \tfrac{1}{\ep}$, hence
under the strategy profile $\xi^*$ the play terminates during the first $K$ kiloblocks
with probability at least $1-\ep$.
By Condition~(F.3), the process $\eta_i$ is a submartingale under the strategy profile $\xi^*$, hence
\[ w_i(y^K) = \eta_i^1 \leq \E_{\xi^*}[\eta_i^{K+1}] \leq \gamma_i(\xi^*) + 2\ep, \]
where the last inequality holds by the choice of $c$, Condition~(A.2''), and since with high probability the play terminates during the first $K$ kiloblocks.

Similarly, if one replaces the plus sign in Eqs.~\eqref{equ:16.1} and~\eqref{equ:16.2} with a minus sign,
the process $\eta_i$ becomes a supermartingale under the strategy profile $\xi^*$, hence
\[ w_i(y^K) = \eta_i^1 \geq \E_{\xi^*}[\eta_i^{K+1}] \geq \gamma_i(\xi^*) - 2\ep, \]
and the claim follows.
\end{proof}

\bigskip

The next lemma, which relies on the definition of $C$, states that
even if some normal player $i \in I_*$ deviates and continues forever, the play terminates with high probability before the end of the $K$'th kiloblock.


\begin{lemma}
\label{lemma:ri}
If $\widehat r^i \not\in \dR^n_{\geq 0}$ for every $i \in I_*$,
then for every player $j \in I_*$ we have
\[ \prob_{(C_j,\xi^*_{-j})}(t_* \hbox{ is smaller than the stage in which kiloblock } K+1 \hbox{ starts}) \geq 1-\ep. \]
\end{lemma}

\begin{proof}
Set $L:=\tfrac{\binom{n}{2}}{\ep}$.
By the choice of $C$ and Condition~(A.1''), there are $1 = k_1 < k_2 < \cdots < k_L < k_{L+1} = K$ such that
\begin{equation}
\label{equ:91}
\sum_{k=k_l}^{k_{l+1}} d(y^k,w(y^k)) > \tfrac{2}{\ep}, \ \ \ \forall l \in [L].
\end{equation}
Call the collection of kiloblocks $\{k \colon k_l \leq k < k_{l+1}\}$ the \emph{$l$'th megablock}.
Eqs.~\eqref{equ:62} and~\eqref{equ:91} implies that under the strategy profile $\xi^*$ the probability that the play terminates during each megablock
is at least $1-\ep$.
We claim that there are at least two normal players who, under strategy profile $\xi^*$, quit during each megablock with probability at least $\ep$.
Indeed, consider the $l$'th megablock and assume by way of contradiction
that there is a unique player $i \in I_*$ who quits with probability larger than $\ep$ during this megablock.
Then $\| w(y^{K-k_l+1}) - \widehat r^i\|_\infty \leq n\ep$.
However, $w(y^{K-k_l+1}) \in \partial D \subseteq \dR^n_{\geq 0}$,
while $\widehat r^i \not\in \dR^n_{\geq 0}$, a contradiction when $\ep$ is sufficiently small.

Since there are $\binom{n}{2}$ pairs of players, the choice of $L$ implies that there is a pair of players
who quit with probability at least $\ep$ in at least $\tfrac{1}{\ep}$ megablocks.
The result follows.
\end{proof}

\bigskip

The next result complete the proof that $\xi^*$ is a sunspot $7\ep$-equilibrium.

\begin{lemma}
\label{lemma:deviation}
For every player $i \in I$ and every pure strategy $\xi_i \in X_i$ we have $\gamma_i(\xi_i,\xi^*_{-i}) \leq w_i(y^K) + 5\ep$.
\end{lemma}

\begin{proof}
Consider first a normal player $i \in I_*$.
We define a stochastic process $\eta_i = (\eta^t_i)_{t \in \dN}$, which approximates the expected continuation payoff of player~$i$ from stage $t$ and on.
Unlike in the proof of Lemma~\ref{lemma:equilibrium}, where the process was defined for kiloblocks, here it is defined for stages.
For each stage $t$ that lies in kiloblock $k$:
\begin{itemize}
\item   If the play was terminated at some stage $t_* < t$, set
\[ \eta_i^t := \widehat r^{S_*}_i - \sum_{l < k} \| y^{K-l+1} - w(y^{K-l}) \|_\infty, \]
where $S_*$ is the set of players who quit at stage $t_*$.
\item   If $t$ lies in a block of type 0 in the $k$'th kiloblock, we set
\[ \eta_i^t := y_i^{K-k+1} - \sum_{l < k} \| y^{K-l+1} - w(y^{K-l}) \|_\infty. \]
\item   If $k$ is the first stage of a block of type~$i$ in the $k$'th kiloblock, we set
\[ \eta_i^t := w_i(y^{K-k+1}) - \sum_{l < k} \| y^{K-l+1} - w(y^{K-l}) \|_\infty. \]
\item   If $k$ is the $l$'th stage of a block of type~$i$ in the $k$'th kiloblock, we set
\[ \eta_i^t := \delta w_i(y^{K-k+1}) + (1-\delta)\widehat r^i_i  - \sum_{l < k} \| y^{K-l+1} - w(y^{K-l}) \|_\infty, \]
where $\delta = (1-\lambda_i(y^{K-k+1}))^{(C_{K-k+1}-l+1)/C_{K-k+1}}$.
\end{itemize}
By Conditions~(F.3) and~(A.2''), the process $(\eta^t_i)_{t \in \dN}$ is a supermartingale under the strategy profile $\xi^*$.
Whenever player~$i$ quits under $\xi^*_i$ with positive probability, he is indifferent between quitting and continuing.
Hence, the process $\eta_i$ is a supermartingale under the strategy profile $(C_i,\xi^*_{-i})$ as well.
Whenever a player other than player~$i$ quits with positive probability,
he does so with probability smaller than $\ep$.
By Assumption~\ref{assumption:1} and since $y^k \in \partial D \subset \dR^n_{\geq 0}$ for every $k\in [K]$,
it follows that for every strategy $\xi_i$ of player~$i$,
\[ w_i(y^K) = \eta_i^1 \geq \E_{(\xi_i,\xi^*_{-i})}[\eta_i^{K+1}] \geq \gamma(\xi_i,\xi^*_{-i}) -5\ep. \]

Consider now an abnormal player $i \not\in I_*$.
Under the strategy profile $\xi^*$, whenever a normal player quits, he does so with probability smaller than $\ep$.
Consequently, if player~$i$ quits at some stage, his expected terminal payoff is bounded by $\ep$.
Let $\eta_i^t$ be the expected payoff of player~$i$ from stage~$t$ and on, assuming that if the game is not terminated by the end of the $K$'th kiloblock,
the continuation payoff is 0.
The process $(\eta_i^t)_{t \in \dN}$ is a martingale that attains nonnegative values before the play terminates at stage $t_*$.
Condition~(A.1'') and Eq.~\eqref{equ:62} imply that the probability that the game is not terminated by the end of the $K$'th kiloblock is smaller than $\ep$,
hence as above
\[ w_i(y^K) = \eta_i^1 \geq \E_{(\xi_i,\xi^*_{-i})}[\eta_i^{K+1}] \geq \gamma(\xi_i,\xi^*_{-i}) -2\ep, \]
and the desired result follows.
\end{proof}

\section{Characterizing the Set of Sunspot Equilibrium Payoffs}
\label{section:characterization}

A vector $x \in \dR^N$ is a \emph{sunspot equilibrium payoff}
if it is the limit of payoffs that correspond to sunspot $\ep$-equilibria, as $\ep$ goes to 0.
Theorem~\ref{theorem:1} proves that if the matrix $\widehat R$ is a $Q$-matrix then there is a sunspot equilibrium payoff in the set $D$.
A complete characterization of the set of sunspot equilibrium payoffs seems to be at present out of reach,
yet it may be possible to characterize the set of sunspot equilibrium payoffs that can be generated by quittings of single players.
In this section we identify one case in which the set of these sunspot equilibrium payoffs can be characterized.

In the literature of linear complementarity problems,
a $Q$-matrix is called an \emph{$M$-matrix} if its diagonal entries are positive and all other entries are nonpositive, see Murty (1988).
Since we require that $r^i_i = 0$ for every $i \in I$, we say that a $Q$-matrix is an \emph{$M$-matrix}
if in every row and every column of the matrix there is exactly one positive entry.

\begin{theorem}
If the matrix $\widehat R$ is an $M$-matrix, then any payoff vector in
$\widetilde D$
is a sunspot equilibrium payoff,
where $\widetilde D := \conv\{r^1,\cdots,r^n\} \cap \dR^N_{\geq 0}$.
\end{theorem}

\begin{proof}
It is well known that any $M$-matrix $\widehat R$ is inverse positive, that is, its inverse $\widehat R^{-1}$ is a nonnegative matrix
(see, e.g., Fujimoto and Ranade, 2004).
Fix $i \in [n]$.
Since $\widehat R^{-1}$ is a nonnegative matrix, $ \widehat R^{-1} e^i \in \dR^n_{\geq 0}$,
where $e^i = (0,\cdots,0,1,0,\cdots,0)$ is the $i$'th unit vector in $\dR^n$.
Therefore we can write
\begin{equation}
\label{equ:61}
\widehat R^{-1} e^i = \sum_{j=1}^n \lambda^i_j e^j,
\end{equation}
where $(\lambda^i_j)_{j=1}^n$ are nonnegative numbers, not all of them zero.
Set
\[ \widehat \lambda^i_j := \frac{\lambda^i_j}{\sum_{j=1}^n \lambda^i_j}. \]
Multiplying both sides of Eq.~(\ref{equ:61}) from the left by $\widehat R$ we get
\[ e^i = \sum_{j=1}^n \lambda_j^i \widehat R e^j = \sum_{j=1}^n \lambda^i_j \widehat r^j, \]
so that $w^i := \tfrac{1}{\| \lambda^i\|_1} e^i \in D$, for each $i \in [n]$.
Both convex hulls $\conv(w^1,\cdots,w^n)$ and $\conv(\widehat r^1,\cdots,\widehat r^n)$ are $(n-1)$-dimensional sets such that
$\conv(w^1,\cdots,w^n) \subseteq \conv(\widehat r^1,\cdots,\widehat r^n)$.
Since in every row of $\widehat R$ there is a single positive entry,
it follows that $D = \conv(w^1,\cdots,w^n)$: every element of $D$ can be presented as a weighted average of $w^1,\cdots,w^n$.

Since $\widehat R$ is an $M$-matrix, there is a unique index $j_i \in [n]$ such that $\widehat r^{j_i}_i > 0$.
Since $w^i$ is a convex combination of $\widehat r^1,\cdots,\widehat r^n$,
since $w^i_i > 0$, and since the unique index $k$ such that $\widehat r^k_i > 0$ is $k=j_i$,
it follows that $\lambda^i_{j_i} > 0$.
Since for all coordinates $k \neq i$ we have $\widehat r^{j_i}_k \leq 0$,
for $\alpha_i > 0$ sufficiently small we have $w^i - \alpha_i \widehat r^{j_i} \in \dR^n_{\geq 0}$,
which implies that $y^{[i]} := \frac{w^i - \alpha_i \widehat r^{j_i}}{1-\alpha_i} \in \conv(w^1,\cdots,w^n)$.
Consequently
$w^i = \alpha_i \widehat r^{j_i} + (1-\alpha_i) y^{[i]}$
and there is a probability distribution $\beta^i = (\beta^i_j)_{j \in [n]}$ such that
$y^{[i]} = \sum_{j=1}^n \beta^i_j w^j$.

The reader can verify that for every $i \in I_*$,
the vectors $(y^{[i]})_{i \in [n]}$, $(w^j)_{j \in [n]}$, and $(\beta^i_j)_{i,j \in [n]}$ satisfy the following conditions,
which are analogous to Conditions (F.1)--(F.5):
\begin{itemize}
\item[(F.1')]    $w^i \in \conv(y^{[i]},\widehat r^i)\setminus \{y^{[i]}\}$ for every $i \in [n]$.
\item[(F.2')]    $w^i_j \geq 0$ for every $i,j \in [n]$.
\item[(F.3')]   $y^{[i]} = \sum_{j=1}^n \beta^i_j w^i$, for every $i \in [n]$.
\item[(F.4')]   If $i \in [n]$ and $\beta^i_j > 0$, then $w^i_i = 0$.
\end{itemize}
This implies that we can repeat the construction in Section~\ref{section:proof:2},
yet now the sequence $(y^\alpha)_{\alpha < \alpha_*}$ has a finite range, namely $(y^{[i]})_{i \in [n]}$:
to implement $y^{[i]}$ as a sunspot equilibrium payoff,
nature chooses $j \in I_*$ according to the probability distribution $(\beta^i_j)_{j \in [n]}$ and the players implement $w^j$.
To implement $w^j$, player~$j$ quits along a sufficiently long block with a total probability of $\lambda_j$,
where $\lambda_j$ satisfies $w^i = \lambda_j \widehat r^j + (1-\lambda_j) y^{[j]}$,
and, if he did not quit,
the players implement the vector $y^{[j]}$.

We note that in this case the ordinal $\alpha_*$ is the first countable ordinal $\omega$.
Since $D = \conv(w^1,\cdots,w^n)$,
by adding an initial stage in which nature chooses which of the vectors $(w^i)_{i \in [n]}$ the players implement as a sunspot $\ep$-equilibrium payoff,
we can implement every vector in $\widetilde D$ as a the payoff of a sunspot $\ep$-equilibrium.
\end{proof}

\bigskip

One  distinction between the construction presented in this section and the construction in the general case is that
while in the general case the sequence $(y^\alpha)_{\alpha < \alpha_*}$ was a deterministic sequence,
here this sequence is a stochastic process, that depends on nature's choices.

\section{Discussion and Open Problems}
\label{sec:discussion}


Quitting games are stopping games in which the payoff processes are constant that are independent of time.
Shmaya and Solan (2004) developed a technique that allows reducing the question of existence of $\ep$-equilibrium in stopping games
with integrable payoff processes
to the question of existence of $\ep$-equilibrium in quitting games or absorbing games.
Heller (2012) and Mashiah-Yaakovi (2014) used the approach of Shmaya and Solan (2004) to prove the existence
of normal-form correlated $\ep$-equilibrium in multiplayer stopping games
and of subgame-perfect $\ep$-equilibrium in multiplayer stopping games with perfect information, respectively.
This approach can be used together with our result to show that every multiplayer stopping game admits a sunspot $\ep$-equilibrium.
The proof is analogous to the proofs in Shmaya and Solan (2004), Heller (2012), and Mashiah-Yaakovi (2014).

Quitting games are also a subclass of absorbing games, which are repeated games in which
player may have more than two actions and there are several nonabsorbing entries (see Vrieze and Thuijsman, 1989).
Absorbing games are a subclass of stochastic games.
The next step in the research is to extend our result to absorbing games, and then to stochastic games.

To construct a sunspot $\ep$-equilibrium from the existence of a solution to the auxiliary game with geometric length
we used a transfinite construction, because we could not prove that the per-stage probability that some player quits in the auxiliary game is
uniformly bounded away from 0.
If the existence of a positive lower bound could be proven, then the construction of a sunspot $\ep$-equilibrium would simplify.

One way to evade the need of a transfinite induction is to show that there is $y^0 \in \partial D$ for which the sequence
$(y^\alpha)_{\alpha < \alpha_*}$ contains finitely many distinct values.
This is what happens when the matrix $\widehat R$ is an $M$-matrix, where the process
$(y^\alpha)_{\alpha < \alpha_*}$ contains $n$ distinct values.
Another case where this phenomenon occurs is when each of the vectors $\widehat r^i$ contains exactly one negative coordinate.
In this case one can show that by properly choosing $y^0$, the sequence $(y^\alpha)_{\alpha < \alpha_*}$ contains two distinct values.
We could not identify an example where the use of a sequence $(y^\alpha)_{\alpha < \alpha_*}$ with infinitely many distinct values is necessary.

When the matrix $\widehat R$ is an $M$-matrix,
the set of sunspot equilibrium payoffs that can be generated by single quittings coincides with the set $\widetilde D$.
We do not know whether this coincidence holds for other classes of $Q$-matrices,
or whether we can provide a different characterization to the set of sunspot equilibrium payoffs that can be generated by our construction,
when the matrix $\widehat R$ is a $Q$-matrix that is not an $M$-matrix.

In our construction, the expected payoff of a player after some history may be negative, albeit at least $-\ep$, see Condition~(F.2) in Theorem~\ref{theorem:2}.
It would be interesting to know whether one can ensure that the expected payoff of all players after every history is nonnegative.

Our study also raises questions about $Q$-matrices, a topic that was not extensively studied in the last decades.
In the characterization of the set of sunspot equilibrium payoffs
in Section~\ref{section:characterization} we used the fact that $M$-matrices are inverse positive.
Unfortunately, the inverse of a $Q$-matrix whose diagonal entries are 0 need not be a nonnegative matrix.
One such example is the $(3\times 3)$-matrix whose off-diagonal entries are equal to 1.
In our application the matrix $\widehat R$ satisfies additional properties, for example, each row and each column contains at least one negative entry.
Is it true that every $Q$-matrix whose diagonal entries are 0 and such that
in each row and each column there is a negative entry is inverse positive?
If not, can one characterize the set of inverse-positive $Q$-matrices that satisfy these conditions?
Suppose that the matrix $\widehat R$ is inverse positive;
can one characterize the set of sunspot equilibria payoffs?

One last question that we will mention concerns the columns of the matrix $\widehat R$.
Suppose that there is a normal player~$i$ such that $r^i_j < 0$ for every $j \neq i$.
Thus, all players prefer to quit alone rather than having player~$i$ quit.
Is the matrix $\widehat R$ a $Q$-matrix?
If not, then it follows that when $\widehat R$ is a $Q$-matrix,
and provided the game admits no stationary $\ep$-equilibrium,
every column of the matrix $\widehat R$ contains at least one positive entry.
This property, if true, may help in future study of quitting games.

\end{document}